\documentclass[11pt,leqno]{article}
\usepackage[margin=1in]{geometry} 
\usepackage{amssymb,amsfonts,amsmath,bbm,mathrsfs,stmaryrd,mathtools}
\usepackage{xcolor}
\usepackage{url}

\usepackage{graphicx}

\usepackage{accents}

\usepackage{extarrows}

\usepackage[shortlabels]{enumitem}
\usepackage{tensor}

\usepackage{xr}
\usepackage[T1]{fontenc}
\usepackage[utf8]{inputenc}

\usepackage[colorlinks,
linkcolor=black!75!red,
citecolor=blue,
pdftitle={},
pdfproducer={pdfLaTeX},
pdfpagemode=UseNone,
bookmarksopen=true,
bookmarksnumbered=true,
backref=page]{hyperref}

\usepackage{tikz}
\usetikzlibrary{arrows,calc,decorations.pathreplacing,decorations.markings,decorations.shapes,intersections,shapes.geometric,through,fit,shapes.symbols,positioning,decorations.pathmorphing}

\makeatletter
\newlength\zig@L
\newlength\zig@La
\newlength\zig@Lb

\newcommand{\xzigrightarrow}[2][]{%
  \mathrel{%
    \settowidth{\zig@La}{$\scriptstyle #2$}%
    \settowidth{\zig@Lb}{$\scriptstyle #1$}%
    \zig@L=\zig@La\relax
    \ifdim\zig@Lb>\zig@L \zig@L=\zig@Lb\fi
    \advance\zig@L by 2.2em\relax
    \tikz[baseline=-0.65ex]{%
      \draw[->,
            line cap=round,
            decorate,
            decoration={zigzag,segment length=4pt,amplitude=1.1pt}]%
        (0,0) -- (\zig@L,0)
        node[midway,above=2pt] {$\scriptstyle #2$}%
        \if\relax\detokenize{#1}\relax\else
          node[midway,below=2pt] {$\scriptstyle #1$}%
        \fi
      ;
    }%
  }%
}
\makeatother

\makeatletter
\newcommand{\squigjoin}{1mu} 

\def\sqleft@{\sim}                    
\def\sqmid@{\sim\mkern-\squigjoin}    

\def\rightsquigarrowfill@{%
  \arrowfill@{\sqleft@}{\sqmid@}{\mkern-4mu\succ}%
}

\newcommand{\xrightsquigarrow}[2][]{%
  \ext@arrow 0359\rightsquigarrowfill@{#1}{#2}%
}
\makeatother

\makeatletter
\newcommand*\circled[1]{\tikz[baseline=(char.base)]{
    \node[shape=circle, draw, inner sep=0pt, 
    minimum height={\f@size},] (char) {\vphantom{WAH1g}#1};}}
\makeatother

\makeatletter
\DeclareRobustCommand\widecheck[1]{{\mathpalette\@widecheck{#1}}}
\def\@widecheck#1#2{%
    \setbox\z@\hbox{\m@th$#1#2$}%
    \setbox\tw@\hbox{\m@th$#1%
       \widehat{%
          \vrule\@width\z@\@height\ht\z@
          \vrule\@height\z@\@width\wd\z@}$}%
    \dp\tw@-\ht\z@
    \@tempdima\ht\z@ \advance\@tempdima2\ht\tw@ \divide\@tempdima\thr@@
    \setbox\tw@\hbox{%
       \raise\@tempdima\hbox{\scalebox{1}[-1]{\lower\@tempdima\box
\tw@}}}%
    {\ooalign{\box\tw@ \cr \box\z@}}}
\makeatother

\usepackage{braket}

\usepackage[amsmath,thmmarks,hyperref]{ntheorem}
\usepackage{cleveref}

\newcommand\nthalias[1]{\AddToHook{env/#1/begin}{\crefalias{lemma}{#1}}}

\nthalias{definition}
\nthalias{example}
\nthalias{examples}
\nthalias{remark}
\nthalias{remarks}
\nthalias{convention}
\nthalias{notation}
\nthalias{construction}
\nthalias{sketch}
\nthalias{theoremN}
\nthalias{propositionN}
\nthalias{corollaryN}
\nthalias{lemma}
\nthalias{proposition}
\nthalias{corollary}
\nthalias{theorem}
\nthalias{conjecture}
\nthalias{question}
\nthalias{assumption}

\creflabelformat{enumi}{#2#1#3}

\crefname{section}{Section}{Sections}
\crefformat{section}{#2Section~#1#3} 
\Crefformat{section}{#2Section~#1#3} 

\crefname{subsection}{\S}{\S\S}
\AtBeginDocument{%
  \crefformat{subsection}{#2\S#1#3}%
  \Crefformat{subsection}{#2\S#1#3}%
}

\crefname{subsubsection}{\S}{\S\S}
\AtBeginDocument{%
  \crefformat{subsubsection}{#2\S#1#3}%
  \Crefformat{subsubsection}{#2\S#1#3}%
}

\theoremstyle{plain}

\newtheorem{lemma}{Lemma}[section]
\newtheorem{proposition}[lemma]{Proposition}

\newtheorem{theorem}[lemma]{Theorem}

\theoremstyle{plain}
\theoremnumbering{Alph}
\newtheorem{theoremN}{Theorem}

\theoremstyle{plain}
\theorembodyfont{\upshape}
\theoremsymbol{\ensuremath{\blacklozenge}}

\newtheorem{definition}[lemma]{Definition}

\newtheorem{remark}[lemma]{Remark}
\newtheorem{remarks}[lemma]{Remarks}

\crefname{definition}{definition}{definitions}
\crefformat{definition}{#2definition~#1#3} 
\Crefformat{definition}{#2Definition~#1#3} 

\crefname{ex}{example}{examples}
\crefformat{example}{#2example~#1#3} 
\Crefformat{example}{#2Example~#1#3} 

\crefname{exs}{example}{examples}
\crefformat{examples}{#2example~#1#3} 
\Crefformat{examples}{#2Example~#1#3} 

\crefname{remark}{remark}{remarks}
\crefformat{remark}{#2remark~#1#3} 
\Crefformat{remark}{#2Remark~#1#3} 

\crefname{remarks}{remark}{remarks}
\crefformat{remarks}{#2remark~#1#3} 
\Crefformat{remarks}{#2Remark~#1#3} 

\crefname{convention}{convention}{conventions}
\crefformat{convention}{#2convention~#1#3} 
\Crefformat{convention}{#2Convention~#1#3} 

\crefname{notation}{notation}{notations}
\crefformat{notation}{#2notation~#1#3} 
\Crefformat{notation}{#2Notation~#1#3} 

\crefname{table}{table}{tables}
\crefformat{table}{#2table~#1#3} 
\Crefformat{table}{#2Table~#1#3}

\crefname{lemma}{lemma}{lemmas}
\crefformat{lemma}{#2lemma~#1#3} 
\Crefformat{lemma}{#2Lemma~#1#3} 

\crefname{proposition}{proposition}{propositions}
\crefformat{proposition}{#2proposition~#1#3} 
\Crefformat{proposition}{#2Proposition~#1#3} 

\crefname{propositionN}{proposition}{propositions}
\crefformat{propositionN}{#2proposition~#1#3} 
\Crefformat{propositionN}{#2Proposition~#1#3} 

\crefname{corollary}{corollary}{corollaries}
\crefformat{corollary}{#2corollary~#1#3} 
\Crefformat{corollary}{#2Corollary~#1#3} 

\crefname{corollaryN}{corollary}{corollaries}
\crefformat{corollaryN}{#2corollary~#1#3} 
\Crefformat{corollaryN}{#2Corollary~#1#3} 

\crefname{theorem}{theorem}{theorems}
\crefformat{theorem}{#2theorem~#1#3} 
\Crefformat{theorem}{#2Theorem~#1#3} 

\crefname{theoremN}{theorem}{theorems}
\crefformat{theoremN}{#2theorem~#1#3} 
\Crefformat{theoremN}{#2Theorem~#1#3} 

\crefname{enumi}{}{}
\crefformat{enumi}{#2#1#3}
\Crefformat{enumi}{#2#1#3}

\crefname{assumption}{assumption}{Assumptions}
\crefformat{assumption}{#2assumption~#1#3} 
\Crefformat{assumption}{#2Assumption~#1#3} 

\crefname{construction}{construction}{Constructions}
\crefformat{construction}{#2construction~#1#3} 
\Crefformat{construction}{#2Construction~#1#3} 

\crefname{sketch}{sketch}{Sketches}
\crefformat{sketch}{#2sketch~#1#3} 
\Crefformat{sketch}{#2Sketch~#1#3} 

\crefname{question}{question}{Questions}
\crefformat{question}{#2question~#1#3} 
\Crefformat{question}{#2Question~#1#3} 

\crefname{equation}{}{}
\crefformat{equation}{(#2#1#3)} 
\Crefformat{equation}{(#2#1#3)}

\numberwithin{equation}{section}

\theoremstyle{nonumberplain}
\theoremsymbol{\ensuremath{\blacksquare}}

\newtheorem{proof}{Proof}
\newcommand\pf[1]{\newtheorem{#1}{Proof of \Cref{#1}}}

\newcommand\bC{{\mathbb C}}

\newcommand\bG{{\mathbb G}}

\newcommand\bR{{\mathbb R}}

\newcommand\bZ{{\mathbb Z}}

\newcommand\cC{{\mathcal C}}

\newcommand\cF{{\mathcal F}}

\newcommand\cL{{\mathcal L}}

\newcommand\cO{{\mathcal O}}

\DeclareMathOperator{\Irr}{Irr}
\DeclareMathOperator{\id}{id}

\DeclareMathOperator{\Spec}{\mathrm{Spec}}

\DeclareMathOperator{\im}{\mathrm{im}}

\DeclareMathOperator{\supp}{\mathrm{supp}}

\newcommand{\qedhere}{\mbox{}\hfill\ensuremath{\blacksquare}}

\newcommand{\comment}[1]{}

\renewcommand{\square}{\mathrel{\Box}}

\title{Uniformity and isotypic smallness for quantum-group representations}
\author{Alexandru Chirvasitu}

\begin{document}

\date{}

\newcommand{\Addresses}{{
  \bigskip
  \footnotesize

  \textsc{Department of Mathematics, University at Buffalo}
  \par\nopagebreak
  \textsc{Buffalo, NY 14260-2900, USA}  
  \par\nopagebreak
  \textit{E-mail address}: \texttt{achirvas@buffalo.edu}

}}

\maketitle

\begin{abstract}
  Compact-group representations on Banach spaces are known to be norm-continuous precisely when they have finite spectra. For a quantum group with continuous-function algebra $\mathcal{C}(\mathbb{G})$ norm continuity can be cast analogously as the bounded weak$^*$-norm continuity of the representation's attached maps $\mathcal{C}(\mathbb{G})^*\to \mathrm{End}(E)$ and its mirror counterpart $E_{\le 1}\times E^*_{\le 1}\to \mathcal{C}(\mathbb{G})$. While the uniformity/isotypic finiteness equivalence no longer holds generally, it does (for the latter map) for compact quantum groups either coamenable or having dimension-bounded irreducible representations. This generalizes the aforementioned classical variant, providing two independent quantum-specific mechanisms of recovering it.  
\end{abstract}


\vspace{.5cm}

\noindent{MSC 2020: 46L67; 20G42; 46A32; 47B01; 54E15; 54A20; 22D25; 22D12
  
}


\section*{Introduction}

The paper revolves around a number of quantum variations on a familiar classical theme: the classification of norm-continuous (or \emph{uniform}, as sometimes termed) compact-group representations $\bG\circlearrowright E$ on Banach spaces as precisely those possessing finitely many \emph{isotypic components} (i.e. \cite[Definition 4.21]{hm5} maximal subrepresentations decomposable as sums of copies of a single irreducible $\bG$-representation). The most direct reference is likely \cite[Corollary 2]{zbMATH05628052}, that paper pointing also to a number of (partial) precursors. Further sampling literature includes \cite{klm_unif} (unitary representations of connected, second-countable locally compact groups) and \cite[Theorem 3.10]{Chirvasitu2026JNCG604} for a number of alternative characterizations of representation uniformity.  

The preceding paragraph's ``quantum'' refers to groups. All such featuring below will be compact quantum groups $\bG$ viewed as in \cite[Definition 3.1]{kt_qg-surv-1} (as the notion has crystallized in the now vast surrounding literature: cf. \cite[Chapter 1]{NeTu13}, \cite[\S 2]{wor-cqg}, etc.): objects dual to their respective non-commutative continuous-function unital $C^*$-algebras $\cC(\bG)$, equipped with
\begin{equation*}
  \begin{gathered}
    \cC(\bG)
    \xrightarrow[\quad\text{coassociative}\quad]{\quad\Delta\quad}
    \cC(\bG)\mathbin{\underline{\otimes}}\cC(\bG)
    \quad\left(\text{\emph{minimal} \cite[Definition IV.4.8]{tak1} $C^*$ tensor product}\right)\\
    \Delta\cC(\bG)\left(1\otimes \cC(\bG)\right)
    ,\quad
    \Delta\cC(\bG)\left(\cC(\bG)\otimes 1\right)
    \quad
    \le
    \quad
    \cC(\bG)^{\underline{\otimes}2}
    \quad\text{dense}.
  \end{gathered}  
\end{equation*}
$\bG$-representations on Banach spaces $E$ are cast in \cite[Definition 3.1]{Chirvasitu2026JNCG604} in imitation of the more familiar \cite[Definition 1.4]{podl_symm} setup of $\bG$-actions on unital $C^*$-algebras: 
\begin{equation*}
  \begin{gathered}
    E
    \xrightarrow[(\Delta\otimes\id)\rho=(\id\otimes\rho)\rho]{\ \rho\ }
    \cC(\bG)\otimes_{\varepsilon}E
    \quad\left(\text{\emph{injective} \cite[Definition A.3.61]{dales_autocont} Banach tensor product}\right)\\
    \left(\cC(\bG)\otimes 1\right)\rho E
    \quad
    \le
    \quad
    \cC(\bG)\otimes_{\varepsilon} E
    \quad\text{dense}.
  \end{gathered}  
\end{equation*}
The familiar Peter-Weyl representation theory developed in \cite[Theorem 1.5]{podl_symm} for $\bG$-actions on $C^*$-algebras then transports over \cite[Theorem 3.2 and surrounding discussion]{Chirvasitu2026JNCG604}, affording continuous idempotents $P_{\rho}^{\alpha}\in \cL(E)$ onto the respective $\alpha$-isotypic components for irreducible representations $\alpha\in\Irr(\bG)$ and hence also a notion of \emph{spectrum} $\Spec \rho:=\left\{\alpha\ :\ P_{\rho}^{\alpha}\ne 0\right\}$. 

Norm continuity too has its quantum counterpart(s): \cite[Definition 3.6(3)]{Chirvasitu2026JNCG604} proposed the weak$^*$-to-norm continuity of the map $\cC(\bG)^*\xrightarrow{(\bullet\otimes \id)\rho}\cL(E)$ on the continuous dual of $\c(\bG)$ attached to $\rho$. That condition indeed being equivalent in full (quantum) generality to isotypic finiteness for fairly simple functional-analytic reasons \cite[Theorem 0.2]{2603.12090v3}, \cite[Definition 1.3]{2603.12090v3} proposes \emph{uniformity$_{\le 1}$} as an alternative: the formally weaker constraint that $(\bullet\otimes \id)\rho$ be weak$^*$-norm continuous on the unit ball $\cC(\bG)^*_{\le 1}$.

We refer to that condition as \emph{$(\bG,E)$-uniformity$_{\le 1}$} instead, to distinguish it from the following left-right mirror counterpart.

\begin{definition}\label{def:chiral.unif}
  A representation $\rho:\bG\circlearrowright E$ of a compact quantum group on a Banach space is \emph{$(E,\bG)$-uniform$_{\le 1}$} if the induced map
  \begin{equation*}
    E_{\le 1}\times E^*_{\le 1}
    \ni
    (v,f)
    \xmapsto{\quad}
    (1\otimes f)\rho v
    \in
    \cC(\bG)
  \end{equation*}
  maps weak$^*$-Cauchy nets into norm-Cauchy nets, the uniformity on the domain being that inherited from the weak$^*$-topology on $\cL(E)^*$ via the embedding $(v,f)\mapsto (T\mapsto fTv)$. 
\end{definition}

It is the latter condition that the paper compares and under appropriate conditions proves equivalent to spectrum finiteness. One such result, obtained in \Cref{th:if.lg.uij.ban}\Cref{item:th:if.lg.uij.ban:cls.pt} as a consequence of a more general principle having to do with a compact quantum group's representative functions' slow rate of decay, reads:

\begin{theoremN}\label{thn:coamnbl}
  Banach-space representations $\bG\circlearrowright E$ of coamenable compact quantum groups are $(E,\bG)$-uniform$_{\le 1}$ precisely when they have finitely many isotypic components.  \qedhere
\end{theoremN}

Recall \cite[Definition 3.1]{bt} that \emph{coamenable} compact quantum groups are those for which the reduced function algebra $\cC_r(\bG)$ has a multiplicative state. Ordinary compact groups in particular being coamenable, \Cref{thn:coamnbl} will suffice to recover its aforementioned classical analogues. Although the statement is not valid in full generality for completely arbitrary quantum groups \cite[Example 1.10]{2603.12090v3}, there are sufficient conditions orthogonal to coamenability that will nevertheless ensure that $(E,\bG)$-uniformity$_{\le 1}$ is equivalent to spectrum finiteness (with the quantum-group class singled out in item \Cref{item:thn:bdd.dim:low} below having received some attention in the literature: \cite[\S 2.3]{dsv}, \cite{MR3871830}). 

\begin{theoremN}\label{thn:bdd.dim}
  \begin{enumerate}[(1),wide]
  \item\label{item:thn:bdd.dim:bdd.sets} Every dimension-bounded spectral subset of a $(E,\bG)$-uniform$_{\le 1}$ compact-quantum-group representation on a Banach space must be finite.
    
  \item\label{item:thn:bdd.dim:low} In particular, the $(E,\bG)$-uniform$_{\le 1}$ Banach-space representations of compact quantum groups with uniformly bounded irreducible representations are precisely those with finite spectrum. 
  \end{enumerate}
\end{theoremN}

Although the boundedness assumption in \Cref{item:thn:bdd.dim:low} will of course not hold generally for classical compact groups $\bG$, \Cref{res:rcvr.cls}\Cref{item:res:rcvr.cls:cls.bdd.rep} notes that \Cref{thn:bdd.dim}\Cref{item:thn:bdd.dim:low} too can be employed in recovering the classical results that motivated the discussion to begin with.


\section{Controlled matrix-coefficient decay and its bearing on uniformity}\label{se:cntrl.dcy}

We assume some familiarity with basic compact-quantum-group formalism, as covered for instance in \cite[Chapter 1]{NeTu13} or \cite[\S 3]{kt_qg-surv-1} (with more specific references provided as needed). A few highlights:
\begin{itemize}[wide]
\item $u^{\alpha}=\left(u^{\alpha}_{ij}\right)_{i,j=1}^{d_{\alpha}:=\dim \alpha}\in \cC(\bG)\otimes M_{d_{\alpha}}$ are unitary elements parametrized by the irreducible representations $\alpha\in \Irr(\bG)$ \cite[Theorem 1.4.3]{NeTu13}, spanning the norm-dense \emph{Hopf $*$-algebra} \cite[Definition 1.6.1]{NeTu13} $\cO(\bG)\le \cC(\bG)$ of matrix coefficients.

\item Said irreducible representations are the (isomorphism classes of) simple $\cO(\bG)$-comodules.

\item $\cC(\bG)\xrightarrow{h=h_{\bG}}\bC$ is the \emph{Haar state} \cite[Theorem 1.2.1]{NeTu13} of $\bG$, analogous to a compact group's Haar probability measure, and faithful on the \emph{reduced version} $\cC_r(\bG)$ of $\cC(\bG)$ by the former's definition as the image of the GNS representation of $h$. 
\end{itemize}

\Cref{def:ban.adapt} builds on the notion of \emph{tempered decay} introduced in passing in the statement of \cite[Theorem 0.3]{2603.12090v3}. That concept is well suited for work in the unitary setup there relevant, and is adapted here to the operative broader Banach-space context.

\begin{definition}\label{def:ban.adapt}
  \begin{enumerate}[(1),wide]
  \item\label{item:def:ban.adapt:xy} Let $(E,\|-\|_E)$, $(X,\|-\|_X)$ and $(Y,\|-\|_Y)$ be Banach spaces with the latter two finite-dimensional. Writing
    \begin{equation*}
      \forall\left(X\xrightarrow[\quad\text{linear}\quad]{\quad u\quad}E\otimes Y\right)
      \ :\ 
      X\otimes Y^*
      \ni
      v\otimes f
      \xmapsto{\quad u_{X,Y}\quad}
      (\id\otimes f)uv
      \in E,
    \end{equation*}
    set 
    \begin{equation*}
      \left\|u\right\|_{X,Y}
      :=
      \left\|u_{X,Y}\right\|
      :\xlongequal[\text{multilinear-map norm}]{\quad\text{\cite[post Theorem A.3.35]{dales_autocont}}\quad}
      \sup_{\substack{f\in Y^*_{\le 1}\\v\in X_{\le 1}}}
      \|(\id\otimes f) u v\|_{E}.      
    \end{equation*}
    We refer to the quantity as the \emph{($X$,$Y$)-norm} of $u$. When $X$ and $Y$ coincide we abbreviate the phrase to \emph{$X$-norm} (and the notation to $u_X$ and $\|u\|_X:=\left\|u_X\right\|$, relying on context to distinguish between the two meanings of $\|-\|_X$). Plainly, $\|-\|_{\bullet}$ are invariant under isometries in the $\bullet$ argument(s).

  \item\label{item:def:ban.adapt:x.univ} If instead $X$ is only a linear space, the \emph{universal} version of $\|u\|_X$ is
    \begin{equation*}
      \|u\|^{\wedge}_X
      :=
      \inf_{(X,\|-\|_X)\in Q(\dim X)}\|u\|_X
    \end{equation*}
    where $Q(d)$ is the \emph{Banach-Mazur compactum} (\cite[post (37.2)]{tj_bm}, \cite[\S 2]{MR1793468}) parametrizing (isometry classes of) $d$-dimensional Banach spaces; recall that the \emph{Banach-Mazur distances}
    \begin{equation}\label{eq:dbm}
      d_{BM}\left((X,\|-\|_X),(Y,\|-\|_Y)\right)
      :=
      \log\inf\left\{\|T\|\cdot \|T^{-1}\|\ :\ X\xrightarrow[\ \text{linear bijection}\ ]{T}Y\right\}
    \end{equation}
    of \cite[\S 2]{MR1793468} (one for each $d$) indeed makes the $Q(\bullet)$ \emph{compacta}, i.e. compact metric spaces.
  \end{enumerate}
\end{definition}

\begin{remark}\label{re:xy.no.go}
  Note a small subtlety concerning \Cref{def:ban.adapt}\Cref{item:def:ban.adapt:x.univ}: the notion would not make sense in the $(X,Y)$ version, whereby one could vary the norms on $X$ and $Y$ independently. Simply scaling those norms would make the infimum identically 0. 
\end{remark}

The requisite language handy, the Banach-flavored \cite[Theorem 0.3]{2603.12090v3} is as follows. 

\begin{theorem}\label{th:if.lg.uij.ban}
  Let $E\xrightarrow{\rho} \cC(\bG)\otimes_{\varepsilon} E$ be a representation of a compact quantum group $\bG$ on a Banach space.
  \begin{enumerate}[(1),wide]
  \item\label{item:th:if.lg.uij.ban:temp.dec} If the Pontryagin dual $\Gamma:=\widehat{\bG}$ \emph{has universal tempered decay} in the sense that
    \begin{equation*}
      \exists\left(C>0\right)
      \forall\left(\alpha\in \Irr(\bG)\right)
      \left(
        \left\|\bC^{\dim\alpha}\xrightarrow{u^{\alpha}}\cC_r(\bG)\otimes \bC^{\dim\alpha}\right\|^{\wedge}_{\bC^{\dim\alpha}}>C
      \right)
    \end{equation*}
    then $\rho$ is $(E,\bG)$-uniform$_{\le 1}$ if and only if it has finite spectrum.
  \item\label{item:th:if.lg.uij.ban:cls.pt} In particular, said equivalence holds if $\bG$ is coamenable.

  \item\label{item:th:if.lg.uij.ban:if.in.tens} Assuming universal tempered decay, $\rho$ also has finite spectrum if and only if it lies in the image of the canonical map
    \begin{equation*}
      \cC(\bG)\otimes_{\varepsilon} \cL(E)
      \xrightarrow{\quad}
      \cL(E,\cC(\bG)\otimes_{\varepsilon}E).
    \end{equation*}
  \end{enumerate}  
\end{theorem}

There is a Banach analogue of \cite[Proposition 1.6]{2603.12090v3} applicable to representations $E\xrightarrow{\rho}\cC(\bG)\otimes_{\varepsilon}E$. \Cref{pr:rho.legs} functions quite broadly, for bounded maps
\begin{equation}\label{eq:rho.pqr}
  E_{p}
  \xrightarrow{\quad\rho=\tensor*[_{p}]{\rho}{_{q}_{r}}\quad}
  E_{q}\otimes_{\varepsilon}E_{r}
  ,\quad
  \text{$E_{\bullet}$ Banach spaces},
\end{equation}
with the indices intended to depict visually which Banach spaces appear on which side of $\rho$. This convention will be handy in depicting the other avatars of $\rho$ featuring in the statement:
\begin{equation*}
  \begin{aligned}
    E_{q}^*
    \ni
    \varphi
    &\xmapsto{\quad\tensor*[_{q}]{\rho}{_{p}_{r}}\quad}
      (\varphi\otimes \id)\rho
      \in\cL(E_{p},E_{r})\\
    E_{p}\times E_{r}^*
    \ni
    (v,\psi)
    &\xmapsto{\quad\tensor*[_{p}_{r}]{\rho}{_{q}}\quad}
      (\id\otimes \psi)\rho v
      \in E_{q}
  \end{aligned}  
\end{equation*}
say, or the analogous $\tensor*[_{r}]{\rho}{_{p}_{q}}$ and $\tensor*[_{p}_{q}]{\rho}{_{r}}$. Recall also \cite[pre \S 1]{MR603371} that a \emph{Cauchy-regular} map between \emph{uniform spaces} \cite[Definition 7.1]{james_unif_1999} is one preserving the Cauchy property for nets. 

\begin{remarks}\label{res:if.in.otimes.eps}
  \begin{enumerate}[(1),wide]
  \item\label{item:res:if.in.otimes.eps:expn} For $E_p:=\bC$ (or $\bR$ if working over the reals) \Cref{eq:rho.pqr} is nothing but an element $\rho\in E_q\otimes_{\varepsilon} E_r$; the notation can thus conveniently be abbreviated to double subscripts, as in
    \begin{equation*}
      E^*_{q}
      \xrightarrow{\quad\tensor*[_q]{\rho}{_r}\quad}
      E_r
      \quad\text{and}\quad
      E^*_{r}
      \xrightarrow{\quad\tensor*[_r]{\rho}{_q}\quad}
      E_q.
    \end{equation*}
    As $\otimes_{\varepsilon}$ is by its very definition the uniform norm restricted from the space of $\bC$-valued functions on the compact Hausdorff space $E^*_{q,\le 1}\times E^*_{r,\le 1}$ (weak$^*$-topologized factors), both maps just-displayed are boundedly weak$^*$-norm continuous: compact Hausdorff spaces are \emph{exponentiable} \cite[Proposition 7.1.5]{brcx_hndbk-2}, so that
    \begin{equation*}
      \begin{aligned}
        \forall
        &\left(Y\text{ compact $T_2$},\ X,Z\text{ topological}\right)\\
        \forall
        &\left(X\times Y\xrightarrow[\text{continuous}]{f}Z\right)
      \end{aligned}
      \left(X\ni x\xmapsto[\text{uniformly continuous}]{\quad}\left(y\xmapsto{\quad} f(x,y)\right)\in \cC(Y,Z)\right).
    \end{equation*}
    Cf. also \cite[post Proposition 2.1]{MR567834} noting the symmetry of that continuity for elements of the injective Banach tensor product. 

  \item The statement of \Cref{pr:rho.legs}\Cref{item:pr:rho.legs:2impl.in.im} references a natural map
    \begin{equation*}
      E\otimes_{\varepsilon} \cL(F,G)
      \ni
      (v,T)
      \xmapsto{\quad}
      \left(w\xmapsto{\quad}v\otimes Tw\right)
      \in      
      \cL(F,E\otimes_{\varepsilon}G)
      ,\quad
      E,F,G\text{ Banach}. 
    \end{equation*}
    That that map (ostensibly defined only on simple tensors) does extend as an isometry (not onto, generally) follows immediately from the selfsame characterization of the $\bullet\otimes_{\varepsilon}\square$ norm as that induced by the function space $\cC\left(\bullet^*_{\le 1}\times\square^*_{\le 1},\bC\right)$ recalled in \Cref{item:res:if.in.otimes.eps:expn}. 
  \end{enumerate}
\end{remarks}

\begin{proposition}\label{pr:rho.legs}
  Consider the following conditions on a continuous linear map \Cref{eq:rho.pqr} for Banach spaces $E_{\bullet}$.

  \begin{enumerate}[(a),wide]
  \item\label{item:pr:rho.legs:q.pr} $\tensor*[_{q}]{\rho}{_{p}_{r}}$ is weak$^*$-norm continuous on the unit ball $E_{q,\le 1}^*$.
  \item\label{item:pr:rho.legs:pr.q} $\tensor*[_{p}_{r}]{\rho}{_{q}}$ is weak$^*$-norm Cauchy-regular on $E_{p,\le 1}\times E_{r,\le 1}^*$, with the weak$^*$ uniformity induced by
    \begin{equation*}
      E_{p}\times E_{r}^*
      \ni
      (v,\psi)
      \xmapsto{\quad}
      \left(T\xmapsto{\quad}\psi Tv\right)
      \in
      \cL(E_p,E_r)^*.
    \end{equation*}
  \end{enumerate}

  \begin{enumerate}[(1),wide]
  \item\label{item:pr:rho.legs:1impl} The implication \Cref{item:pr:rho.legs:pr.q} $\Rightarrow$ \Cref{item:pr:rho.legs:q.pr} holds.

  \item\label{item:pr:rho.legs:2impl.in.im} Both conditions hold if $\rho$ belong to the image of the canonical map
    \begin{equation*}
      E_q\otimes_{\varepsilon} \cL(E_p,E_r)      
      \xrightarrow{\quad}
      \cL\left(E_p,E_q\otimes_{\varepsilon} E_r\right).
    \end{equation*}
  \end{enumerate}  
\end{proposition}
\begin{proof}
  \begin{enumerate}[label={},wide]

  \item\textbf{\Cref{item:pr:rho.legs:1impl}} Assume a net
    \begin{equation*}
      E^*_{q,\le 1}
      \ni
      \varphi_{\lambda}
      \xrightarrow[\quad\lambda\quad]{\quad\text{weak$^*$}\quad}
      0
      ,\quad
      \forall \lambda
      \bigg(
      \left\|\left(\varphi_{\lambda}\otimes \id\right)\rho\right\|> C>0
      \bigg)
      ,\quad
      \text{fixed $C$}.
    \end{equation*}
    This ensures the existence of $v_{\lambda}\in E_{p,\le 1}$ and $\psi_{\lambda}\in E^*_{r,\le 1}$ with
    \begin{equation}\label{eq:bdd.above.c}
      \forall\lambda
      \bigg(
      \left|
        \psi_{\lambda}
        \left(\varphi_{\lambda}\otimes \id\right)\rho
        v_{\lambda}
        =
        \left(\varphi_{\lambda}\otimes \psi_{\lambda}\right)\rho
        v_{\lambda}
      \right|
      >C
      \bigg).
    \end{equation}
    We can furthermore assume, upon passing to a subnet if necessary, that $\left(v_{\lambda},\psi_{\lambda}\right)_{\lambda}$ is weak$^*$-Cauchy. The regularity assumption then forces the norm Cauchy property on $\left(\left(\id\otimes \psi_{\lambda}\right)\rho v_{\lambda}\right)_{\lambda}$, whence that net's convergence to some $w\in E_q$. The boundedness of $\left(\varphi_{\lambda}\right)_{\lambda}$ and its weak$^*$ 0-convergence then jointly imply
    \begin{equation*}
      \left(\varphi_{\lambda}\otimes \psi_{\lambda}\right)\rho
      v_{\lambda}
      =
      \varphi_{\lambda}
      \left(\id\otimes \psi_{\lambda}\right)\rho
      v_{\lambda}
      \xrightarrow[\quad\lambda\quad]{\quad}
      0,
    \end{equation*}
    contradicting \Cref{eq:bdd.above.c}. 

  \item\textbf{\Cref{item:pr:rho.legs:2impl.in.im}}: Immediate from \Cref{res:if.in.otimes.eps}\Cref{item:res:if.in.otimes.eps:expn}.
  \end{enumerate}
\end{proof}

Note incidentally that \cite[Lemma 1.8]{2603.12090v3} too has a Banach-space variant, consequent on \Cref{pr:rho.legs} just as the former follows from \cite[Proposition 1.6]{2603.12090v3}. For a compact-quantum-group representation $E\xrightarrow{\rho}\cC(\bG)\otimes_{\varepsilon}E$ on a Banach space we define \emph{supports} for elements of $E$ and $E^*$ respectively by
\begin{equation*}
  \begin{aligned}
    \Irr(\bG)
    \supseteq 
    \supp v
    &:=
      \left\{
      \alpha\in \Irr(\bG)
      \ :\
      P^{\alpha}v\ne 0
      \right\}
      \quad\text{and}
    \\
    \Irr(\bG)
    \supseteq 
    \supp f
    &:=
      \bigcap\left\{\cF'\subseteq \Irr(\bG)\ :\ \forall\left(\alpha\in \Irr(\bG)\setminus \cF'\right)\left(f|_{\im P^{\alpha}}=0\right)\right\}.  
  \end{aligned}  
\end{equation*}

\begin{lemma}\label{le:ban.rep.legs}
  A representation $E\xrightarrow{\rho}\cC(\bG)\otimes_{\varepsilon}E$ of a compact quantum group on a Banach space is uniform$_{\le 1}$ if
  \begin{equation*}
    \left(\supp f\to\infty \wedge \supp w\to\infty\right)
    \ 
    \xRightarrow{\quad}
    \
    (\id\otimes f)\rho w
    \xrightarrow[\quad\text{in $\cC(\bG)$}\quad]{\quad\text{norm}\quad}
    0,
  \end{equation*}
  where $f\in E^*_{\le 1}$, $w\in E_{\le 1}$.  \qedhere
\end{lemma}

\pf{th:if.lg.uij.ban}
\begin{th:if.lg.uij.ban}
  \begin{enumerate}[label={},wide]
  \item\textbf{\Cref{item:th:if.lg.uij.ban:temp.dec}} Assume infinitely many spectral projections $P^{\alpha}$ of $\rho$ non-zero.  This then provides non-vanishing ``Fourier coefficients''
    \begin{equation*}
      x^{\alpha}_{ij}
      :=
      \left(\psi^{\alpha}_{ij}\otimes \id \right)\rho\in \cL(E)
      ,\quad
      1\le i,j\le d_{\alpha}:=\dim\alpha
      ,\quad
      \left(u^{\beta}_{k\ell}\xmapsto{\ \psi^{\alpha}_{ij}\ }\delta_{\alpha\beta}\delta_{ik}\delta_{j\ell}\right)
      \in \cC(\bG)^*.
    \end{equation*}
    Fixing an $\alpha$ for the moment and focusing on $x_{ij}:=x^{\alpha}_{ij}$, said operators act as matrix units in the sense that $x_{ij}x_{k\ell}=\delta_{jk}x_{i\ell}$; they thus operate as the usual rank-1 matrix units on
    \begin{equation*}
      E':=
      \bigoplus_{i=1}^{d_{\alpha}}\bC e_i
      \le \left(\text{isotypic component }E_{\alpha}\right)
      \le E
      ,\quad
      e_1\in \im x_{11}\text{ arbitrary and }
      e_j:=x_{j1}e_1.
    \end{equation*}
    Identify
    \begin{equation*}
      \sum_{i,j} u^{\alpha}_{ij}\otimes x_{ij}
      \in
      \cC(\bG)\otimes \cL\left(E'\right)
      \cong
      \cL\left(E',\cC(\bG)\otimes E'\right)
    \end{equation*}
    with $u^{\alpha}\in \cC(\bG)\otimes \cL\left(\bC^{d_{\alpha}}\right)$. The $d_{\alpha}$-space $E'$ comes equipped with its Banach structure via $E'\le E$, and the tempered-decay estimate now ensures the existence of norm-$(\le 1)$
    \begin{equation*}
      f_{\alpha}\in E'^*_{\le 1}
      ,\quad
      w_{\alpha}\in E'_{\le 1}
      ,\quad
      \left\|\left(\id\otimes f_{\alpha}\right)\rho w_{\alpha}\right\|>C
    \end{equation*}
    which may as well be regarded as members of $E^*_{\alpha,\le 1}$ (by extension via Hahn-Banach) and $E_{\alpha,\le 1}$ respectively. Now
    \begin{itemize}[wide]
    \item the $x_{\alpha}:=\left(\id\otimes f_{\alpha}\right)\rho w_{\alpha}$ will cluster at some norm-$(\ge C)$ $x\in \cC(\bG)$ by Alaoglu and the assumed $(E,\bG)$-uniformity$_{\le 1}$;
    \item while on the other hand $x_\alpha\in \sum_{ij}\bC u^{\alpha}_{ij}$ and hence $h(x^*_{\alpha}x_{\alpha'})=0$ for $\alpha\ne \alpha'$ \cite[Theorem 1.4.3]{NeTu13}; the Haar state $h$ being faithful on $\cC_r(\bG)$ \cite[Corollary 1.7.5]{NeTu13}, we have a contradiction. 
    \end{itemize}

  \item\textbf{\Cref{item:th:if.lg.uij.ban:cls.pt}} The state $\varepsilon$ assigns value 1 to all diagonal matrix coefficients $u^{\alpha}_{ii}$, so these all have norm $\ge 1$ in $\cC_r(\bG)$ (exactly 1 in fact, given that $\sum_j u^{\alpha *}_{ij} u^{\alpha}_{ij}=1$). Thus:
    \begin{equation*}
      \left\|(\id\otimes f_{\alpha})u^{\alpha}v_{\alpha}\right\|=1
      ,\quad
      \forall
      \left(
        \begin{gathered}
          v_{\alpha}\in \bC e_i\le \bC^{\dim\alpha},\ \|v_{\alpha}\|=1\\
          f_{\alpha}\in \bC e_i^*\le \left(\bC^{\dim\alpha}\right)^*,\ \|f_{\alpha}\|=1
        \end{gathered}
      \right)
    \end{equation*}
    for a basis $\left(e_j\right)_{j=1}^{\dim\alpha}$ compatible with the matrix units $u^{\alpha}_{jk}$ and any Banach-space structure on $\bC^{\dim\alpha}$ (with the corresponding dual norm on the dual space $\left(\bC^{\dim\alpha}\right)^*$), hence the hypothesis of the just-proven part \Cref{item:th:if.lg.uij.ban:temp.dec} and the conclusion.

  \item\textbf{\Cref{item:th:if.lg.uij.ban:if.in.tens}} follows from \Cref{item:th:if.lg.uij.ban:temp.dec} and \Cref{pr:rho.legs}\Cref{item:pr:rho.legs:2impl.in.im}.  \qedhere
  \end{enumerate}
\end{th:if.lg.uij.ban}

The preceding proof makes clear precisely how \Cref{th:if.lg.uij.ban} relies on working with the \emph{reduced} version $\cC_r(\bG)$: the faithfulness of the Haar state is invoked. One way to dispense with that constraint is to isolate the precise large-norm condition that will function on arbitrary $\cC(\bG)$.

\begin{definition}\label{def:dstl}
  A discrete quantum group $\Gamma=\widehat{\bG}$ is \emph{$\cC(\bG)$-distal}\footnote{In terminology borrowed from the dynamical-systems literature: \cite[p.401]{zbMATH03151299}, \cite[p.732]{zbMATH03595892}, etc.: the phrase is meant to convey failure to cluster.} if for every net $\left(\alpha_{\lambda}\right)_{\lambda}\subset\Irr(\bG)$ eventually leaving every finite subset of $\Irr(\bG)$ and arbitrary Banach-space structures on the carrier spaces $V_{\alpha_{\lambda}}$ of the respective $\bG$-representations there are
  \begin{equation*}
    v_{\alpha_{\lambda}}\in V_{\alpha_{\lambda},\le 1}
    \quad\text{and}\quad
    f_{\alpha_{\lambda}}\in V^*_{\alpha_{\lambda},\le 1}
  \end{equation*}
  with $\left(\left(\id\otimes f_{\alpha_{\lambda}}\right) u^{\alpha_{\lambda}} v_{\alpha_{\lambda}}\right)_{\lambda}\subset \cC(\bG)$ having no (norm-)Cauchy subnets. 
\end{definition}

The proof of \Cref{th:if.lg.uij.ban} then in fact adapts to yield the following sufficient condition that will ensure finite-spectrum/uniformity$_{\le 1}$ equivalence.

\begin{theorem}\label{th:g.dist}
  If $\bG$ is a compact quantum group with $\cC(\bG)$-distal dual a Banach-space representation $E\xrightarrow{\rho}\cC(\bG)\otimes_{\varepsilon} E$ is uniform$_{\le 1}$ if and only if it has finite spectrum.  \qedhere
\end{theorem}

We turn next to the Introduction's second narrative branch.

\pf{thn:bdd.dim}
\begin{thn:bdd.dim}
  That \Cref{item:thn:bdd.dim:bdd.sets} implies \Cref{item:thn:bdd.dim:low} is immediate, so only the former need detain us. Fix, to that end, a Banach-space representation $E\xrightarrow{\rho} \cC(\bG)\otimes_{\varepsilon} E$ and a spectral subset $\cF\subseteq \Spec(\rho)$ thereof, dimension-bounded in the sense that
  \begin{equation*}
    \sup\left\{d_{\alpha}:=\dim {\alpha}\ :\ \alpha\in \cF\right\}<\infty.
  \end{equation*}
  The goal being to prove $\cF$ finite (assuming uniformity$_{\le 1}$), one may as well further assume all $d_{\alpha}$ equal to a common $d\in \bZ_{\ge 0}$.

  Consider now, for each $\alpha\in \cF$, a copy of $\alpha$ in $\rho$, supported on the $d$-dimensional Banach space $E_{\alpha}\le E$. $\rho$ then (co)restricts to $E_{\alpha}\xrightarrow{\rho_{\alpha}} \cC(\bG)\otimes E_{\alpha}$, and the finite diameter of the Banach-Mazur compactum $Q(d)$ under the Banach-Mazur distance of \Cref{eq:dbm} ensures the existence of $C>0$ and orthonormal bases for $\alpha\in \cF$ which yield
  \begin{equation*}
    \forall\left(\alpha\in \cF,\ 1\le i,j\le d=d_{\alpha}\right)
    \left(u^{\alpha}_{ij}\in \im\left(E_{\alpha,\le C}\times E^*_{\alpha,\le C}\xrightarrow{(\id\otimes\bullet)\rho_{\alpha}\bullet}\cC(\bG)\right)\right).
  \end{equation*}
  The bounded weak$^*$-norm Cauchy regularity of $(\id\otimes \bullet)\rho\bullet$ provided by uniformity$_{\le 1}$ in conjunction with \Cref{pr:rho.legs} thus ensures the total norm boundedness of the family $\left\{u^{\alpha}_{ij}\ :\ \alpha\in \cF,\ 1\le i,j\le d_{\alpha}\right\}$ of matrix coefficients. We can then assume $u^{\alpha}_{ij}$ simultaneously cluster with varying $\alpha$, uniformly in $1\le i,j\le d$, to matrix coefficients $u_{ij}$ of a matrix subcoalgebra $C\le \cC(\bG)$:
  \begin{equation*}
    u_{ij}
    \xmapsto{\quad\Delta\quad}
    \sum_{k=1}^d u_{ik}\otimes u_{k j}
    ,\quad
    \text{unitary }
    u:=(u_{ij})_{i,j=1}^d\in \cC(\bG)\otimes M_d. 
  \end{equation*}
  This suffices to conclude that $u$ is one of the $u^{\alpha}$; this being so for \emph{arbitrary} cluster points of $\left\{u^{\alpha}\right\}_{\alpha\in \cF}$, $\cF$ must be finite. 
\end{thn:bdd.dim}

\begin{remarks}\label{res:rcvr.cls}
  \begin{enumerate}[(1),wide]
  \item\label{item:res:rcvr.cls:cls.bdd.rep} For classical $\bG$ the dimension bounding required by \Cref{thn:bdd.dim}\Cref{item:thn:bdd.dim:low} holds precisely \cite[Theorem 1]{MR302817} when $\bG$ is \emph{virtually abelian} (i.e. has a finite-index abelian subgroup). Classical uniformity/spectrum finiteness equivalence can nevertheless be recovered from said item \Cref{item:thn:bdd.dim:low}:
    \begin{itemize}[wide]
    \item a norm-continuous Banach-space representation $\bG\to GL(E)$ has compact image, which is thus \cite[Theorem 9.3.14]{2602.12362v1} a (finite-dimensional) compact Lie subgroup.
      
    \item Having reduced the problem to compact Lie groups, standard weight theory \cite[Theorem VI.2.10]{btd_lie_1995} further distills it to its version for tori (so abelian groups). 
    \end{itemize}
    In reference to reducing norm continuity to simpler classes of groups (abelian, profinite, etc.) see also the multiple (classical) characterizations of norm continuity in \cite[Theorem 3.10]{Chirvasitu2026JNCG604}.

  \item\label{item:res:rcvr.cls:alt.pf.nobm} There is (at least) one alternative approach to proving \Cref{thn:bdd.dim}\Cref{item:thn:bdd.dim:bdd.sets} that does not entail invoking Banach-Mazur compactness and any attendant norm estimates.

    Observe first that (per the proof of \Cref{th:if.lg.uij.ban}\Cref{item:th:if.lg.uij.ban:cls.pt}) the diagonal matrix coefficients $u^{\alpha}_{ii}$, $\alpha\in \Spec\rho$ are in any case contained in the image of $(\id\otimes\bullet)\rho\bullet$ and hence constitute a totally norm-bounded family by uniformity$_{\le 1}$, regardless of any dimension-bounding constraints. With the additional assumption that
    \begin{equation*}
      \forall\left(\alpha\in \cF\subseteq \Irr(\bG)\right)
      \left(\dim\alpha=d\right)
    \end{equation*}
    for a fixed on-negative integer $d$, one can further suppose (\cite[Theorem 1.4.3(i)]{NeTu13}, for instance) that for some fixed $1\le i\le d$ all $h\left(u^{\alpha*}_{ii}u^{\alpha}_{ii}\right)$, $\alpha\in \cF$ dominate $\frac 1d$. Since on the other hand $u^{\alpha,\alpha'}_{ii}$ are $h$-orthogonal \cite[Theorem 1.4.3(ii)]{NeTu13} for distinct $\alpha,\alpha'\in \cF$, this contradicts said coefficients' clustering even in the $L^2$ norm induced by $h$, let alone norm-wise.
  \end{enumerate}
\end{remarks}


\addcontentsline{toc}{section}{References}

\begin{thebibliography}{10}

\bibitem{MR1793468}
Sergei~M. Ageev and Du{\v{s}}an Repov{\v{s}}.
\newblock On {B}anach-{M}azur compacta.
\newblock {\em J. Austral. Math. Soc. Ser. A}, 69(3):316--335, 2000.

\bibitem{zbMATH03595892}
Joseph Auslander and Shmuel Glasner.
\newblock Distal and highly proximal extensions of minimal flows.
\newblock {\em Indiana Univ. Math. J.}, 26:731--749, 1977.

\bibitem{bt}
E.~B\'{e}dos and L.~Tuset.
\newblock Amenability and co-amenability for locally compact quantum groups.
\newblock {\em Internat. J. Math.}, 14(8):865--884, 2003.

\bibitem{brcx_hndbk-2}
Francis Borceux.
\newblock {\em Handbook of categorical algebra. 2: {Categories} and
  structures}, volume~51 of {\em Encycl. Math. Appl.}
\newblock Cambridge: Univ. Press, 1994.

\bibitem{btd_lie_1995}
Theodor Br{\"o}cker and Tammo tom Dieck.
\newblock {\em Representations of compact {Lie} groups. {Corrected} reprint of
  the 1985 orig}, volume~98 of {\em Grad. Texts Math.}
\newblock New York, NY: Springer, corrected reprint of the 1985 orig. edition,
  1995.

\bibitem{Chirvasitu2026JNCG604}
Alexandru Chirv{\u{a}}situ.
\newblock {(Quantum) discreteness, spectrum compactness and uniform
  continuity}.
\newblock {\em {Journal of Noncommutative Geometry}}, 20(1):69--94, 2026.

\bibitem{2603.12090v3}
Alexandru Chirvasitu.
\newblock Spectral finiteness, quantum norm continuity and classical points,
  2026.
\newblock \url{http://arxiv.org/abs/2603.12090v3}.

\bibitem{dales_autocont}
H.~G. Dales.
\newblock {\em Banach algebras and automatic continuity}, volume~24 of {\em
  London Mathematical Society Monographs. New Series}.
\newblock The Clarendon Press, Oxford University Press, New York, 2000.
\newblock Oxford Science Publications.

\bibitem{dsv}
Matthew Daws, Adam Skalski, and Ami Viselter.
\newblock Around property ({T}) for quantum groups.
\newblock {\em Comm. Math. Phys.}, 353(1):69--118, 2017.

\bibitem{zbMATH03151299}
Robert Ellis.
\newblock Distal transformation groups.
\newblock {\em Pac. J. Math.}, 8:401--405, 1958.

\bibitem{2602.12362v1}
Helge Gloeckner and Karl-Hermann Neeb.
\newblock {Infinite-Dimensional} {Lie} {Groups}, 2026.
\newblock \url{https://arxiv.org/abs/2602.12362v1}.

\bibitem{hm5}
Karl~H. Hofmann and Sidney~A. Morris.
\newblock {\em The structure of compact groups. {A} primer for the student. {A}
  handbook for the expert}, volume~25 of {\em De Gruyter Stud. Math.}
\newblock Berlin: De Gruyter, 5th edition edition, 2023.

\bibitem{james_unif_1999}
Ioan James.
\newblock {\em Topologies and uniformities. {Exp}. and rev. version of
  {Topological} and uniform spaces, 1987}.
\newblock Springer Undergrad. Math. Ser. London: Springer, exp. and rev.
  version of {Toplogical} and uniform spaces, 1987 edition, 1999.

\bibitem{klm_unif}
R.~R. Kallman.
\newblock A characterization of uniformly continuous unitary representations of
  connected locally compact groups.
\newblock {\em Mich. Math. J.}, 16:257--263, 1969.

\bibitem{MR3871830}
Jacek Krajczok and Piotr~M. So\l~tan.
\newblock Compact quantum groups with representations of bounded degree.
\newblock {\em J. Operator Theory}, 80(2):415--428, 2018.

\bibitem{kt_qg-surv-1}
Johan Kustermans and Lars Tuset.
\newblock A survey of {$C^*$}-algebraic quantum groups. {I}.
\newblock {\em Irish Math. Soc. Bull.}, 43:8--63, 1999.

\bibitem{MR302817}
Calvin~C. Moore.
\newblock Groups with finite dimensional irreducible representations.
\newblock {\em Trans. Amer. Math. Soc.}, 166:401--410, 1972.

\bibitem{NeTu13}
Sergey Neshveyev and Lars Tuset.
\newblock {\em Compact quantum groups and their representation categories},
  volume~20 of {\em Cours Sp\'ecialis\'es [Specialized Courses]}.
\newblock Soci\'et\'e Math\'ematique de France, Paris, 2013.

\bibitem{podl_symm}
Piotr Podle\'{s}.
\newblock Symmetries of quantum spaces. {S}ubgroups and quotient spaces of
  quantum {${\rm SU}(2)$} and {${\rm SO}(3)$} groups.
\newblock {\em Comm. Math. Phys.}, 170(1):1--20, 1995.

\bibitem{zbMATH05628052}
A.~I. Shtern.
\newblock Norm continuous representations of locally compact groups.
\newblock {\em Russ. J. Math. Phys.}, 15(4):552--553, 2008.

\bibitem{MR603371}
Ray~F. Snipes.
\newblock Cauchy-regular functions.
\newblock {\em J. Math. Anal. Appl.}, 79(1):18--25, 1981.

\bibitem{tak1}
M.~Takesaki.
\newblock {\em Theory of operator algebras. {I}}, volume 124 of {\em
  Encyclopaedia of Mathematical Sciences}.
\newblock Springer-Verlag, Berlin, 2002.
\newblock Reprint of the first (1979) edition, Operator Algebras and
  Non-commutative Geometry, 5.

\bibitem{tj_bm}
Nicole Tomczak-Jaegermann.
\newblock {\em Banach-{Mazur} distances and finite-dimensional operator
  ideals}, volume~38 of {\em Pitman Monogr. Surv. Pure Appl. Math.}
\newblock Harlow: Longman Scientific \&| Technical; New York: John Wiley \&|
  Sons, Inc., 1989.

\bibitem{MR567834}
Jun Tomiyama.
\newblock Inner derivations in the tensor products of operator algebras.
\newblock {\em Tohoku Math. J. (2)}, 32(1):91--97, 1980.

\bibitem{wor-cqg}
S.~L. Woronowicz.
\newblock Compact quantum groups.
\newblock In {\em Sym\'{e}tries quantiques ({L}es {H}ouches, 1995)}, pages
  845--884. North-Holland, Amsterdam, 1998.

\end{thebibliography}



\Addresses

\end{document}